\documentclass{article}
\usepackage[utf8]{inputenc}

\usepackage{amsmath}
\usepackage{amsfonts}
\usepackage{mathrsfs}
\usepackage{mathtools}
\usepackage[dvipdfmx]{graphicx}
\usepackage{here}
\usepackage{amsthm}
\usepackage{algorithmic}
\usepackage{algorithm}
\usepackage{url}
\usepackage{stmaryrd}
\usepackage[title]{appendix}
\usepackage{amssymb}
\usepackage[all]{xy}
\usepackage{color}
\usepackage{extarrows}
\usepackage{multirow}
\usepackage{enumitem}
\usepackage[
        hypertexnames=false,
        hyperindex,
        pagebackref,
        breaklinks=true,
        bookmarks=false,
        colorlinks,
        linkcolor=blue,
        citecolor=blue,
        urlcolor=green,
]{hyperref}

\newtheorem{theorem}{Theorem}[subsection]
\newtheorem{lemma}[theorem]{Lemma}
\newtheorem{proposition}[theorem]{Proposition}

\theoremstyle{definition}

\newtheorem{example}[theorem]{Example}
\newtheorem{remark}[theorem]{Remark}

\theoremstyle{plain}
\newtheorem{theor}{Theorem}
\theoremstyle{plain}
\newtheorem{coro}{Corollary}
\numberwithin{equation}{subsection}

\title{Explicit construction of a plane sextic model\\ for genus-five Howe curves, II
}
\date{\today}
\author{Momonari Kudo}
\begin{document}


%


\maketitle

\begin{abstract}
    A {\it Howe curve} is defined as the normalization of the fiber product over a projective line of two hyperelliptic curves.
    Howe curves are very useful to produce important classes of curves over fields of positive characteristic, e.g., maximal, superspecial, or supersingular ones.
    Determining their feasible equations explicitly is a basic problem, and it has been solved in the hyperelliptic case and in the non-hyperelliptic case with genus not greater than $4$.
    In this paper, we construct an explicit plane sextic model for non-hyperelliptic Howe curves of genus $5$.
    We also determine the number and type of singularities on our sextic model, and prove that the singularities are generically $4$ double points.
    Our results together with Moriya-Kudo's recent ones imply that for each $s \in \{2,3,4,5\}$, there exists a non-hyperellptic curve $H$ of genus $5$ with $\mathrm{Aut}(H) \supset \mathbf{V}_4$ such that its associated plane sextic has $s$ double points.  
\end{abstract}

\section{Introduction}\label{sec:intro}

{Throughout this paper, we measure the (computational) complexity by the number of arithmetic operations in a field.}
Let $k$ be an algebraically closed field of characteristic $p$ with $p =0$ or $p \geq 5$, and $\mathbb{P}^n$ the projective $n$-space over $k$.
By a curve, we mean a (possibly singular) projective variety of dimension one over $k$.
In this paper, we focus on an important class of curves that are birational to fiber products of hyperelliptic curves, and consider the problem of computing their explicit plane models.
The Jacobian variety of such a curve can be decomposed via isogenies into a product of low-dimensional abelian varieties or the Jacobian varieties of curves of lower genus (cf.\ \cite[Section 3.3]{glass2005hyperelliptic}).
Moreover, such a curve over a finite field is expected to have many rational points with respect to genus, see \cite{howe2016quickly} for the case of genus $4$, and \cite{howe2017curves} for the case of genera $5$, $6$, and $7$.
In \cite{katsura2021decomposed}, Katsura-Takashima named a smooth curve birational to the fiber product of two hyperelliptic curves a {\it generalized Howe curve}.
The precise definition is as follows:
For two hyperelliptic curves $C_1$ and $C_2$ of genera $g_1$ and $g_2$ over $k$ with $0<g_1 \leq g_2$, let $\psi_1 : C_1 \to \mathbb{P}^1$ and $\psi_2 : C_2 \to \mathbb{P}^1$ be usual double covers (hyperelliptic structures), and assume that there is no isomorphism $\varphi: C_1 \to C_2$ with $\psi_2 \circ \varphi = \psi_1$.
Then, we call the desingularization $H$ of the fiber product $C_1 \times_{\mathbb{P}^1} C_2$ over these hyperellitic structures a generalized Howe curve associated with $C_1$ and $C_2$, and in this paper we refer to it simply as a {\it Howe curve}. 
The genus $g$ of $H$ is calculated as $g = 2 (g_1+g_2)+1-r$, where $r$ is the number of ramification points in $\mathbb{P}^1$ common to $C_1$ and $C_2$.
Here, if $g\geq 4$, then $H$ is non-hyperelliptic if and only if $r<g_1+g_2+1$, see \cite[Theorem 2]{katsura2021decomposed}.
For example, a non-hyperelliptic Howe curve of genus $5$ has the parameter set $(g_1,g_2,r) = (2,2,4)$, $(1,2,2)$, $(1,3,4)$, or $(1,1,0)$; in fact, it suffices to consider $(2,2,4)$ and $(1,1,0)$ only, see Section \ref{subsec:defHowe} below for details.

Howe curves are also useful to generate supersingular or superspecial curves (cf.\ \cite{oort1991hyperelliptic}, \cite{kudo2020existence}, \cite{kudo2020algorithms}, \cite{ohashi2022fast}, \cite{moriya2022some}), where a non-singular curve is said to be supersingular (resp.\ superspecial) if its Jacobian variety is isogenous (resp.\ isomorphic) to a product of supersingular elliptic curves.
Isomorphism classes of supersingular or superspecial curves can be determined by using their explicit equations.
More generally, in terms of the classification of curves, it is important to find explicit equations for curves, in particular plane models.
As for Howe curves, if $H$ defined as above is hyperelliptic, then its hyperelliptic equation can be easily computed, see e.g., \cite[Section 2]{moriya2022some}.
In the case where $H$ is non-hyperelliptic, Moriya-Kudo~\cite[Section 4.2]{moriya2022some} found a plane quartic equation for genus-$3$ Howe curves explicitly.
For the case of genus $4$, we refer to \cite[Section 2]{kudo2020algorithms}, where the authors represent $H$ as the complete intersection of a quadric and a cubic in $\mathbb{P}^3$, which is the canonical model of a non-hyperelliptic curve of genus $4$.
Recently, Moriya-Kudo~\cite{MK23} also obtained a plane sextic model for non-hyperelliptic Howe curves of genus $5$ specified by the parameter set $(g_1,g_2,r) = (2,2,4)$, whereas the remaining case $(g_1,g_2,r) = (1,1,0)$ is left unsolved.

In this paper, we determine an explicit plane sextic model for $H$ with $(g_1,g_2,r) = (1,1,0)$.
More specifically, we shall prove the following theorem:

\begin{theor}\label{thm:main1}
    Every non-hyperelliptic Howe curve $H$ of genus five associated with two genus-$1$ curves $C_1 : y_1^2 = (x-\alpha_1)(x-\alpha_2)(x-\alpha_3)(x-\alpha_4)$ and $C_2: y_2^2 = (x- \beta_1)(x-\beta_2)(x-\beta_3)(x-\beta_4)$ with $\alpha_i,\beta_j \in k$ is birational to a plane sextic singular curve $C : f (x,y)= 0$ defined by
    \[
    \begin{aligned}
    f=& (c_{60} x^6 + c_{42} x^4 y^2) + (c_{50} x^5 + c_{32} x^3 y^2) + (c_{40} x^4 + c_{22}x^2 y^2 + c_{04}y^4)\\
    & + (c_{30} x^3 + c_{12}x y^2) + (c_{20} x^2 + c_{02} y^2) + c_{10} x + c_{00},
    \end{aligned}
    \]
    where $c_{ij}$'s are written as polynomials in $\alpha_1, \alpha_2,\alpha_3,\alpha_4,\beta_1, \beta_2, \beta_3, \beta_4$ explicitly.
    Moreover, once $\alpha_1, \alpha_2,\alpha_3,\alpha_4,\beta_1, \beta_2, \beta_3, \beta_4$ are given, all the $13$ coefficients $c_{ij}$ can be computed \textcolor{black}{with operations in a field to which all $\alpha_i$'s and $\beta_j$'s belong, and the number of required operations in the field is bounded by a constant (not depending on the characteristic $p$ of $k$).}
\end{theor}

Unlike a method by Moriya-Kudo~\cite{MK23} for the case $(g_1,g_2,r)=(2,2,4)$, we do not use any resultant computation, but a simpler construction that generalizes Katsura-Takashima's example~\cite[Example 4]{katsura2021decomposed}.
This construction can be extended to the case of Howe curves of other genus, see Remark \ref{rem:gen} below.

We also classify singularities of our plane sextic curve $C$ in Theorem \ref{thm:main1}.
Denoting by $\tilde{C}$ the projective closure in $\mathbb{P}^2$ of $C$, we obtain the following:

\begin{theor}\label{thm:main2}
    The projective curve $\tilde{C}$ has exactly $2$, $3$, or $4$ singularities, all of which are of multiplicity $2$.
    Moreover, it has $4$ double points generically.
\end{theor}

By these theorems and \cite[Theorems 1 and 2]{MK23}, we obtain the following:

\begin{coro}
For each $s \in \{2,3,4,5\}$, there exists a non-hyperelliptic non-singular curve $H$ of genus $5$ with $\mathrm{Aut}(H) \supset \mathbf{V}_4$ such that $H$ is birational to a plane singular sextic having $s$ double points, where $\mathbf{V}_4$ is the Klein $4$-group.
\end{coro}

Our sextic model constructed in this paper would be feasible to analyze non-hyperelliptic Howe curves of genus five as plane singular curves.
The use of this model could have further applications such as enumerating isomorphism classes of genus-$5$ non-hyperelliptic Howe curves with specific properties (e.g., superspecial), see Section \ref{sec:conc} below for details.

\section{Preliminaries}\label{sec:pre}

In this section, we first collect some known facts on Howe curves.
We also determine possible factorization patterns of a bivariate sextic, which are used in the next section.

\subsection{Generalized Howe curves}\label{subsec:genHowe}

Let $k$ be an algebraically closed field of characteristic $p$ with $p =0$ or $p \geq 5$.
Let $C_1$ and $C_2$ be hyperelliptic curves of genera $g_1$ and $g_2$ over $k$ with $0<g_1 \leq g_2$, and let $\psi_1: C_1 \to \mathbb{P}^1$ and $\psi_2 : C_2 \to \mathbb{P}^1$ be their hyperelliptic structures.
With mutually distinct elements $\alpha_i, \beta_j \in k \cup \{ \infty \}$ for $1 \leq i \leq 2g_1+2$ and $1 \leq j \leq 2g_2+2$, we can write
\[
\begin{aligned}
    C_1 : y_1^2 = \phi_1(x) = (x-\alpha_1) \cdots (x- \alpha_r) (x-\alpha_{r+1}) \cdots (x- \alpha_{2g_1+2}),\\
    C_2 : y_2^2 =\phi_2 (x)= (x-\alpha_1) \cdots (x- \alpha_r) (x-\beta_{r+1}) \cdots (x- \beta_{2g_2+2}).
\end{aligned}
\]
In the case where $\alpha_i = \infty$ or $\beta_j= \infty$, we remove the factors $x-\alpha_i$ or $x-\beta_j$ from the equations for $C_1$ and $C_2$.
Supposing the non-existence of an isomorphism $\varphi : C_1 \to C_2$ with $\psi_2 \circ \varphi = \psi_1$, the fiber product $C_1 \times_{\mathbb{P}^1} C_2$ over $\mathbb{P}^1$ is irreducible, and $r \leq g_1+g_2+1$.
In \cite{katsura2021decomposed}, Katsura-Takashima define a {\it generalized Howe curve} $H$ associated with $C_1$ and $C_2$ as a non-singular curve birational to the (possibly singular) curve $C_1 \times_{\mathbb{P}^1} C_2$.
In this paper we refer to it simply as a {\it Howe curve}. 
They also proved in \cite[Proposition 1]{katsura2021decomposed} that $g(H)=2(g_1+g_2) +1 - r$, where $g(H)$ denotes the genus of $H$.
Note that the hyperelliptic involutions of $C_1$ and $C_2$ lift to automorphisms $\sigma_1$ and $\sigma_2$ on $H$ of order-$2$, and therefore the automorphism group $\mathrm{Aut}(H)$ contains a subgroup isomorphic to the Klein $4$-group $\mathbf{V}_4$.
We obtain the quotient curve $C_3 := H/ \langle \sigma_1 \sigma_2 \rangle$ of $H$ by $\sigma_1 \sigma_2$, 
The genus of $C_3$ is $g_3 = g_1+g_2+1-r$, and a defining equation is
\[
C_3 : y_3^2 = \phi_3(x) = (x-\alpha_{r+1}) \cdots (x- \alpha_{2g_1+2}) (x-\beta_{r+1}) \cdots (x- \beta_{2g_2+2}).
\]
It is also proved in \cite{katsura2021decomposed} that, if $g(H) \geq 4$, then $H$ is hyperelliptic if and only if $r = g_1+g_2+1$, i.e., $C_3$ is rational.


\subsection{Defining equations for Howe curves}\label{subsec:defHowe}

Howe curves are interesting objects in themselves, and also useful to produce curves (over finite fields) having many rational points with respect to genus, see \cite{howe2016quickly} and \cite{howe2017curves}.
See also \cite{kudo2020algorithms}, \cite{kudo2020existence}, \cite{moriya2022some}, and \cite{ohashi2022fast} for applications to construct supersingular or superspecial curves.
Constructing defining equations for Howe curves is a basic problem.
In the case where $H$ is hyperelliptic, a formula of its hyperelliptic equation is known, see \cite[Lemmas 2.2.5 and 2.4.2]{moriya2022some}.
In the case where $H$ is non-hyperelliptic, Moriya-Kudo~\cite[Section 4.2]{moriya2022some} found an explicit plane quartic model for genus-$3$ Howe curves.
For the case of genus $4$, Kudo-Harashita-Howe~\cite[Section 2]{kudo2020algorithms} represented $H$ as the complete intersection of a quadric and a cubic in $\mathbb{P}^3$, which is the canonical model of a non-hyperelliptic curve of genus $4$.

This paper studies genus-$5$ non-hyperelliptic Howe curves.
As it was noted in Section \ref{sec:intro}, such a curve is constructed from parameter sets $(g_1,g_2,r) = (2,2,4)$, $(1,2,2)$, $(1,3,4)$, or $(1,1,0)$.
Among them, when $(g_1,g_2,r) = (1,2,2)$, the third curve $C_3$ defined in the previous subsection has genus $g_3 = 2$, and it shares exactly $r'= 4$ ramification points with $C_2$.
Since $C_2 \times_{\mathbb{P}^1} C_3$ is birational to $C_1 \times_{\mathbb{P}^1} C_2$, we find that considering the case $(g_1,g_2,r) = (1,2,2)$ is equivalent to considering the case $(g_1,g_2,r) = (2,2,4)$.
Similarly, for $(g_1,g_2,r) = (1,3,4)$, one can verify that $C_1$ and $C_3$ with $(g_1,g_3) = (1,1)$ has no common ramification point, and that $C_1 \times_{\mathbb{P}^1} C_3$ is birational to $C_1 \times_{\mathbb{P}^1} C_2$.
Therefore, it suffices to consider the cases $(g_1,g_2,r) = (2,2,4)$ and $(1,1,0)$ only.
The former case was first studied in \cite{howe2017curves} by Howe with his motivation to produce curves over finite fields with many rational points, and its plane sextic model is explicitly constructed in \cite{MK23} by Moriya-Kudo:

\begin{theorem}[{\cite[Theorems 1 and 2]{MK23}}]\label{thm:MK1-1}
    With notation as above, every non-hyperelliptic Howe curve $H$ of genus $5$ associated with two genus-$2$ curves $C_1 : y^2 = x (x-1)(x-\alpha_1)(x-\alpha_2)(x-\alpha_3)$ and $C_2:y^2 = x(x-1)(x-\alpha_1)(x-\beta_2)(x-\beta_3)$ with $\alpha_i,\beta_j \in k \smallsetminus \{ 0,1\}$ is birational to a plane sextic singular curve $C$ defined by
    \[
    f = \sum_{0< i+j \leq 3} c_{2i,2j}Y^{2i}Z^{2j} = 0
    \]
    with a node $(0,0)$, where $c_{2i,2j}$'s are written as polynomials in $\alpha_1, \alpha_2, \alpha_3, \beta_2, \beta_3$ explicitly.
    Moreover, once $\alpha_1, \alpha_2, \alpha_3, \beta_2,\beta_3$ are given, all the $9$ coefficients $c_{2i,2j}$ can be computed with operations in a field to which all $\alpha_i$'s and $\beta_j$'s belong, and the number of required operations in the field is bounded by a constant (not depending on the characteristic $p$ of $k$).

    In addition, the projective closure $\tilde{C}$ in $\mathbb{P}^2$ of $C$ has exactly $3$ or $5$ singularities, all of which are of multiplicity $2$.
    Furthermore, $\tilde{C}$ has $5$ double points generically.
\end{theorem}

This paper focuses on the latter case: $(g_1,g_2,r)=(1,1,0)$.
As in the former case, we will construct a plane sextic model in Section \ref{sec:main} below.

\subsection{Factorization patterns of a certain bivariate sextic}
Let $k$ be an arbitrary field of characteristic different from $2$.
Consider a bivariate sextic in $k[x,y]$ of the form
\[
\begin{aligned}
f=& (b_{60} x^6 + b_{42} x^4 y^2) + (b_{50} x^5 + b_{32} x^3 y^2) + (b_{40} x^4 + b_{22}x^2 y^2 + b_{04}y^4)\\
& + (b_{30} x^3 + b_{02}x y^2) + (b_{20} x^2 + b_{02} y^2) + b_{10} x + b_{00}
\end{aligned}
\]
with $b_{42} = -4$ and $b_{04} = 1$, where $b_{ij}$ denotes the coefficient of $x^i y^j$ in $f$.
We also assume that $b_{60} = b_{40} = b_{20} = b_{00} = 0$ does not hold, namely $f(x,0) \neq 0$.

\begin{lemma}\label{lem:factor}
With notation as above, assume that $f$ is reducible (over $k$).
Then $f$ is factored into the product of $H_1$ and $H_2$ given as follows:
\begin{enumerate}
    \item[{\bf (A)}] For $a_i\in k$ with $1 \leq i \leq 7$,
    \[
    \begin{cases}
    H_1 = y^2 + (a_1 x^2 + a_2 x + a_3),\\
    H_2 = y^2 + (- 4 x^4 + a_4 x^3 + a_5 x^2 + a_6 x +  a_7),
    \end{cases}
    \]
    whence
    \[
    \begin{aligned}
        H_1H_2=& (-4 a_1 x^6 - 4 x^4 y^2) + ((a_1 a_4 - 4a_2)x^5 + a_4 x^3 y^2) \\
        & + ((a_1 a_5 + a_2 a_4 - 4 a_3 ) x^4 + (a_1 + a_5) x^2 y^2 + y^4)\\
        & + ((a_1 a_6 + a_2 a_5 + a_3 a_4) x^3 + (a_2 + a_6) x y^2) \\
        & + ((a_1 a_7 + a_2 a_6 + a_3 a_5 ) x^2  + (a_3 + a_7) y^2) + (a_2 a_7 + a_3 a_6)x + a_3 a_7.
    \end{aligned}
    \]
    
    \item[{\bf (B)}] For $a_i \in k$ with $1 \leq i \leq 6$,
    \[
    \begin{cases}
        H_1 = y^2 + (2 x^2 + a_1 x + a_2) y + (a_3 x^3 + a_4 x^2 + a_5 x + a_6),\\
        H_2 = y^2 + (-2 x^2 - a_1 x - a_2) y + (a_3 x^3 + a_4 x^2 + a_5 x + a_6),
    \end{cases}
    \]
    whence
    \[
    \begin{aligned}
        H_1H_2=& (a_3^2 x^6 - 4 x^4 y^2) + ((2a_3 a_4 x^5 + (-4a_1 + 2 a_3) x^3 y^2) \\
        & + ((2 a_3 a_5 + a_4^2 ) x^4 + ( - a_1^2 - 4 a_2 + 2 a_4) x^2 y^2 + y^4)\\
        & + ((2 a_3 a_6 + 2 a_4 a_5 ) x^3 + (-2a_1a_2 + 2a_5) x y^2) \\
        & + ((2a_4 a_6 +  a_5^2 ) x^2  + (-a_2^2 + 2a_6) y^2) + 2 a_5 a_6 x + a_6^2.
    \end{aligned}
    \]
\end{enumerate}
\end{lemma}

\begin{proof}
    We first consider the case where $f$ has a linear factor, say $ a y + b x + c$ with $a,b,c \in k$.
    If $a = 0$, then $b \neq 0$, and $f(-c/b,y)$ is identically zero, but this contradicts $b_{04} \neq 0$.
    Therefore, we have $a \neq 0$, whence we may assume $a=1$.
    Since $f$ is a polynomial in $y^2$, it is divided by both of $y + b x + c$ and $y - (b x + c)$.
    Here, we claim $b x + c \neq - (b x + c)$.
    Indeed, if $b x + c = - b x - c$, then $b = c = 0$, which implies that $f$ is divisible by $y$.
    However, this contradicts our assumption $f(x,0) \neq 0$.
    Thus, $f$ is divisible by $H_1 := (y + b x + c)(y - (b x + c))$.
    Putting $H_1 := y^2 + (a_1 x^2 + a_2 x + a_3)$ for $a_1,a_2,a_3 \in k$ and dividing $f$ by it, we \textcolor{black}{obtain} $f = H_1 H_2$, where $H_2$ is of the form given in the statement {\bf (A)}.

    Next, we assume that $f$ has no linear factor.
    In this case, we can factorize $f$ into $f = H_1 H_2$ for irreducible polynomials $H_1$ and $H_2$ in $k[x,y]$ with $(\mathrm{deg}(H_1), \mathrm{deg}(H_2)) =(2,4)$ or $(3,3)$, or into $f = Q_1 Q_2 Q_3$ for irreducible quadratic polynomials $Q_1$, $Q_2$, and $Q_3$ in $k[x,y]$.
    It will be proved in Lemma \ref{lem:3quad} below that $f$ is written as the form of {\bf (A)} in the latter case, and we here consider the former case only.
    We may also suppose from $b_{04} =1$ that $H_1$ and $H_2$ are monic polynomials of degree $\geq 1$ with respect to $y$.
    Note that
    \[
    H_1(x,y) H_2(x,y) = H_1(x,-y)H_2(x,-y)
    \]
    from $f(x,y) = f(x,-y)$, whence $(H_1(x,y),H_2(x,y)) = (H_1(x,-y),H_2(x,-y))$ or $(H_1(x,y),H_2(x,y)) =(H_2(x,-y),H_1(x,-y))$ holds from the irreducibilities of $H_1$ and $H_2$.

    In the case where $(\mathrm{deg}(H_1), \mathrm{deg}(H_2)) =(2,4)$, we may assume that $H_1$ and $H_2$ are given as
        \[
    \left\{
    \begin{aligned}
        H_1 =& y + (a_{1} x^2 + a_{2} x + a_{3}),\\
        H_2 =& y^3 + (a_4 x^2 + a_5 x + a_6) y^2 + (a_7 x^3 + a_8 x^2 + a_9 x + a_{10})y \\
        & + (a_{11} x^4 + a_{12} x^3 +a_{13} x^2 + a_{14}x + a_{15}) 
    \end{aligned}
    \right.
    \]
    or
    \[
    \left\{
    \begin{aligned}
        H_1 =& y^2 + (a_1 x + a_2) y + (a_3 x^2 + a_4 x + a_5),\\
        H_2 =& y^2 + (a_6 x^3 + a_7 x^2 +a_8 x + a_9) y + (a_{10}x^4 +a_{11}x^3 + a_{12} x^2 + a_{13} x + a_{14})
    \end{aligned}
    \right.
    \]
    since $f$ is monic in degree $4$ as a polynomial in $y$ over $k[x]$.
    The former case is impossible.
    Indeed, since $H(x,-y)$ also divides $f$, we deduce $H_1 = y$, which contradicts $f(x,0) \neq 0$.
    Therefore, we may suppose the latter case.
and we have $(H_1(x,y),H_2(x,y)) = (H_1(x,-y),H_2(x,-y))$ from $(\mathrm{deg}(H_1), \mathrm{deg}(H_2)) =(2,4)$.
In this case, it follows that $a_1=a_2=a_6=a_7=a_8=a_9=0$ and $a_{10}=-4$, and therefore
\[
    \begin{cases}
    H_1 = y^2 + (a_3 x^2 + a_4 x + a_5),\\
    H_2 = y^2 + (- 4 x^4 + a_{11} x^3 + a_{12} x^2 + a_{13} x +  a_{14}).
    \end{cases}
\]
Renaming indices of $a_i$'s, we obtain the desired $H_1$ and $H_2$.

    In the case where $(\mathrm{deg}(H_1), \mathrm{deg}(H_2)) =(3,3)$, since $f$ is monic in degree $4$ as a polynomial in $y$ over $k[x]$, we may assume that $H_1$ and $H_2$ are given as
    \[
    \left\{
    \begin{aligned}
        H_1 =& y^3 + (a_1 x + a_2) y^2 + (a_3 x^2 + a_4 x + a_5)y + (a_6 x^3 + a_7 x^2 +a_8 x + a_9) ,\\
        H_2 =& y + (a_{10}x^3 + a_{11} x^2 + a_{12} x + a_{13}),
    \end{aligned}
    \right.
    \]
    or
    \[
    \left\{
    \begin{aligned}
        H_1 =& y^2 + (a_1 x^2 + a_2 x + a_3)y + (a_4 x^3 + a_5 x^2 +a_6 x + a_7) ,\\
        H_2 =& y^2 + (a_8 x^2 + a_9 x + a_{10})y + (a_{11} x^3 + a_{12} x^2 +a_{13} x + a_{14}).
    \end{aligned}
    \right.
    \]
    The former case is impossible.
    Indeed, the coefficient of $x^3 y^3$ in $H_1H_2$ is equal to $a_{10}$, which should be zero by $b_{33}=0$ in $f$.
    On the other hand, the coefficient of $x^4 y^2$ is equal to $a_1 a_{10} = 0$, which contradicts our assumption $b_{42} \neq 0$.
    Here, let us consider the latter case.
    If $(H_1(x,y),H_2(x,y)) = (H_1(x,-y),H_2(x,-y))$, then $a_1 = a_2 = a_3=0$, so that the coefficient of $x^4 y^2$ in $H_1H_2$ is zero, a contradiction.
    Therefore, one has $(H_1(x,y),H_2(x,y)) = (H_2(x,-y),H_1(x,-y))$, and thus $H_1$ and $H_2$ are given as
     \[
    \left\{
    \begin{aligned}
        H_1 =& y^2 + (a_1 x^2 + a_2 x + a_3)y + (a_4 x^3 + a_5 x^2 +a_6 x + a_7) ,\\
        H_2 =& y^2 + (-a_1 x^2 - a_2 x - a_{3})y + (a_{4} x^3 + a_{5} x^2 +a_{6} x + a_{7}),
    \end{aligned}
    \right.
    \]
    with $a_1^2 = 4$.
    It is straightforward that we may assume $a_1 =2$, and we obtain $H_1$ and $H_2$ as in {\bf (B)}, by renaming indices of $a_i$'s.

\end{proof}

\begin{lemma}\label{lem:3quad}
    With notation as above, if $f$ is factored into the product of three irreducible quadratic polynomials $Q_1$, $Q_2$, and $Q_3$ in $k[x,y]$, then it is written as the form {\bf (A)}.
\end{lemma}

\begin{proof}
     Since $f$ is a monic quartic in $y$, we may suppose that $Q_1$, $Q_2$, and $Q_3$ are monic polynomials of degree $\geq 1$ with respect to $y$, and may also assume that $(\deg_y Q_1, \deg_y Q_2,\deg_y Q_3) = (2,1,1)$.
    We can write $Q_2 = y + (ax^2 + b x + c)$ and $Q_3 =  y + (a'x^2 + b' x + c')$ for $a,b,c,a',b',c' \in k$.
        By a discussion similar to the first paragraph in the proof of Lemma \ref{lem:factor} together with the irreducibilities of $Q_2$ and $Q_3$, we obtain $a' = -a$, $b' = -b$, and $c' = -c$, so that $f$ is divisible by 
        \[
        Q_2Q_3 = y^2 + (b x^4 + a_4 x^3 + a_5 x^2 + a_6 x +  a_7)
        \]
        for some $b,a_4,a_5,a_6,a_7 \in k$.
        The condition $b_{42} = -4$ implies $b=-4$, so that $Q_2Q_3$ is of the same form as in $H_2$ for {\bf (A)}.
        Dividing $f$ by $Q_1Q_2$, we also \textcolor{black}{find} from $\deg Q_1 =2$ that $Q_1$ is of the same form as in $H_1$ for {\bf (A)}, as desired.
\end{proof}

\section{Our plane sextic model}\label{sec:main}

In this section, we shall explicitly construct a plane sextic model for non-hyperelliptic Howe curves of genus $5$ associated with two genus-$1$ double covers sharing no ramification point.
Let $k$ be an algebraically closed field of characteristic $p$ with $p =0$ or $p \geq 5$.

\subsection{Construction}\label{subsec:const}

Let $H$ be a non-hyperelliptic Howe curve of genus $5$ over $k$ associated with genus-$1$ curves $C_1$ and $C_2$ sharing no ramification point.
Considering the M\"{o}bius transformation on ramification points, we may assume that each of $C_1$ and $C_2$ is not ramified at $\infty$.
For pairwise distinct elements $\alpha_1,\alpha_2,\alpha_3,\alpha_4,\beta_1,\beta_2,\beta_3,\beta_4$ in $k$, we can write
\begin{eqnarray*}
    & & C_1  :  y_1^2  =  \phi_1 (x) = (x-\alpha_1)(x-\alpha_2)(x-\alpha_3)(x-\alpha_4),\\
    & & C_2 :  y_2^2  = \phi_2 (x) = (x-\beta_1)(x-\beta_2)(x-\beta_3)(x-\beta_4),\\
    & & C_3 :  y_3^2  = \phi_1(x) \phi_2(x),
\end{eqnarray*}
where the third curve $C_3$ is defined in Subsection \ref{subsec:genHowe}.
While Moriya-Kudo~\cite{MK23} construct a plane sextic model in the case $(g_1,g_2,r)=(2,2,4)$ by computing resultants, we realize it by a simpler method, which is a generalization of Katsura-Takashima's example provided in \cite[Example 4]{katsura2021decomposed}.
Specifically, putting $y=y_1+y_2$ and squaring both sides, we obtain $2 y_1 y_2 = y^2 - (\phi_1 + \phi_2)$.
Taking the square of both sides again, we have $4 \phi_1 \phi_2 = y^4 - 2(\phi_1 + \phi_2) y^2 +(\phi_1 + \phi_2)^2 $, so that
\begin{equation}\label{eq:sextic}
    f := y^4 - 2 (\phi_1 + \phi_2) y^2 + (\phi_1 - \phi_2)^2 = 0,
\end{equation}
where $\phi_1-\phi_2$ has degree $\leq 3$.
Since $\phi_1 + \phi_2$ is a quartic with $x^4$-coefficient $2$, the polynomial $f$ is a sextic with $x^4 y^2$-coefficient $-4$.

\begin{proposition}
    With notation as above, if $f$ is absolutely irreducible, then $C_1 \times_{\mathbb{P}^1} C_2$ is birational to the (singular) affine curve $C : f(x,y) = 0$, whence $H$ is isomorphic to the normalization of the projective closure $\tilde{C}$ in $\mathbb{P}^2$ of $C$.
\end{proposition}

\begin{proof}
    Clearly $\Phi : C_1 \times_{\mathbb{P}^1} C_2 \dashrightarrow C \ ; \ (x,y_1,y_2) \mapsto (x,y_1+y_2)$ is a well-defined rational map.
    Conversely, the inverse rational map $\Psi : C \dashrightarrow C_1 \times_{\mathbb{P}^1} C_2$ can be constructed as follows:
    Let $(x,y)$ be a point on $C$.
    It follows from
    \[
    \begin{aligned}
        f =&  \left(y^2 - \left( \phi_1 + \phi_2 + 2 \sqrt{\phi_1\phi_2} \right) \right)\left(y^2 - \left( \phi_1 + \phi_2 - 2 \sqrt{\phi_1\phi_2} \right) \right)\\
        = & \left(y^2 - (\sqrt{\phi_1} + \sqrt{\phi_2})^2 \right) \left(y^2 - (\sqrt{\phi_1} - \sqrt{\phi_2})^2 \right) 
    \end{aligned}
    \]
    that we can write $y= \varepsilon_1 \sqrt{\phi_1(x)} + \varepsilon_2 \sqrt{\phi_2(x)}$ for $\varepsilon_i \in \{ -1,1\}$ uniquely.
    Then, clearly the point $(x, \varepsilon_1 \sqrt{\phi_1(x)}, \varepsilon_2 \sqrt{\phi_2(x)})$ lies on $C_1 \times_{\mathbb{P}^1} C_2$, and it is straightforward that $\Phi \circ \Psi = \mathrm{id}_C$ and $\Psi \circ \Phi = \mathrm{id}_{C_1 \times_{\mathbb{P}^1} C_2}$ as rational maps.
    It is also straightforward that the number of points at infinity of $C$ is finite, whence $H$ is isomorphic to the normalization of $\tilde{C}$, as desired.
\end{proof}

To prove the absolute irreducibility and to investigate singularities of $\tilde{C}$, let us write down $f$ more concretely.
For each integer $i$ with $1\leq i \leq 4$, we denote by $\sigma_i$ and $\tau_i$ the degree-$i$ elementary symmetric polynomial on $\alpha_1,\alpha_2,\alpha_3,\alpha_4$ and that on $\beta_1,\beta_2,\beta_3,\beta_4$ respectively, say
\[
\begin{array}{llll}
\sigma_1 := \displaystyle\sum_{i} \alpha_i, & \sigma_2 := \displaystyle\sum_{i<j} \alpha_i \alpha_j, & \sigma_3 := \displaystyle\sum_{i<j<k} \alpha_i \alpha_j \alpha_k, & \sigma_4 := \alpha_1\alpha_2\alpha_3\alpha_4,\\
\tau_1 := \displaystyle\sum_{i} \beta_i, & \tau_2 := \displaystyle\sum_{i<j} \beta_i \beta_j, & \tau_3 := \displaystyle\sum_{i<j<k} \beta_i \beta_j \beta_k, & \tau_4 := \beta_1\beta_2\beta_3\beta_4.
\end{array}
\]
It is straightforward that
\begin{eqnarray*}
    \phi_1 &=& x^4 - \sigma_1 x^3 + \sigma_2 x - \sigma_3 x + \sigma_4,\\
    \phi_2 &=& x^4 - \tau_1 x^3 + \tau_2 x - \tau_3 x + \tau_4,\\
    \phi_1 + \phi_2 &=& 2 x^4 - \left( \sigma_1 + \tau_1 \right) x^3 + \left( \sigma_2 + \tau_2 \right)x^2 - \left( \sigma_3 + \tau_3 \right)x + (\sigma_4 + \tau_4), \\
    \phi_1 - \phi_2 &=& - \left( \sigma_1 - \tau_1 \right) x^3 + \left( \sigma_2 - \tau_2 \right)x^2 - \left( \sigma_3 - \tau_3 \right)x + (\sigma_4 - \tau_4) .
\end{eqnarray*}
Here, we can write
\[
\begin{aligned}
f=& (c_{60} x^6 + c_{42} x^4 y^2) + (c_{50} x^5 + c_{32} x^3 y^2) + (c_{40} x^4 + c_{22}x^2 y^2 + c_{04}y^4)\\
& + (c_{30} x^3 + c_{12}x y^2) + (c_{20} x^2 + c_{02} y^2) + c_{10} x + c_{00}
\end{aligned}
\]
with $c_{42} = -4$ and $c_{04} = 1$, and {the other coefficients are computed as follows:}
\begin{align}
  c_{60} =& (\sigma_1 - \tau_1)^2,\label{eq:c60}\\
    c_{50} =&  -2 (\sigma_1 - \tau_1) \cdot \left(\sigma_2 - \tau_2\right),\\ 
  c_{32} = & 2(\sigma_1 + \tau_1),\\  
     c_{40} =& 2 \left(\sigma_1 - \tau_1 \right)\left(\sigma_3 - \tau_3 \right)+ \left(\sigma_2-\tau_2 \right)^2, \label{eq:c40} \\
  c_{22} = & {-}2\left(\sigma_2 + \tau_2\right), \\ 
  c_{30} = & -2 (\sigma_1 - \tau_1) (\sigma_4-\tau_4)  - 2  (\sigma_2 - \tau_2 ) {(\sigma_3 - \tau_3)},\\ 
  c_{12} =& 2  \left( \sigma_3 + \tau_3  \right),\label{eq:c12} \\
   c_{20} =&  {(\sigma_3 -\tau_3)^2} + 2 (\sigma_4-\tau_4) (\sigma_2-\tau_2), \\
  c_{02} =& - 2 (\sigma_4+\tau_4),\\
   c_{10} =& - 2 (\sigma_4-\tau_4) \left( \sigma_3 -\tau_3 \right), \label{eq:c10}\\
  c_{00} =& (\sigma_4-\tau_4)^2. \label{eq:c00}
\end{align}

\begin{remark}\label{rem:3points}
    Considering M\"{o}bius transformations, we can fix arbitrary $3$ elements among the $8$ elements $\alpha_i$'s and $\beta_j$'s, e.g., $(\alpha_1,\alpha_2,\alpha_3) = (0,1,-1)$.
    This implies that the non-hyperelliptic Howe curves of genus-$5$ associated with two genus-$1$ curves form a family of at most $5$ dimension in the moduli space of curves of genus $5$.
\end{remark}

\begin{remark}\label{rem:gen}
    The above construction of $f$ works for Howe curves of other genus associated with hyperelliptic curves sharing no ramification points, but the absolute irreducibility of $f$ should be checked in each case.
\end{remark}

\subsection{Proof of the irreducibility}

Here, we shall prove that the sextic $f$ constructed in the previous subsection is absolutely irreducible.
Note that $c_{60} = c_{40} = c_{20} = c_{00}=0$ does not hold, since $\alpha_i$'s and $\beta_j$'s are mutually distinct.

\begin{proposition}\label{prop:irr}
With notation as above, $f$ is irreducible over $k$ for every $(\alpha_1, \alpha_2,\alpha_3,\alpha_4, \beta_1, \beta_2,\beta_3, \beta_4)$ with pairwise distinct elements $\alpha_i$'s and $\beta_j$'s in $k$.
\end{proposition}
   
\begin{proof}
If $f$ were reducible, then we could factor it into the product of two polynomials $H_1$ and $H_2$ in $k[x,y]$ as in {\bf (A)} or {\bf (B)} of Lemma \ref{lem:factor}.
It suffices to prove the irreducibility in the case where one of $\alpha_i$'s and $\beta_j$'s is equal to $0$, for example $\alpha_1 = 0$.
Indeed, we have $f(x+\alpha_1,y) = y^4 - 2(\overline{\phi}_1 + \overline{\phi}_2) y^2 + (\overline{\phi}_1 - \overline{\phi}_2)^2$ for $\overline{\phi}_1=\phi_1(x+\alpha_1)$ and $\overline{\phi}_2=\phi_2(x+\alpha_1)$, which implies that $f(x+\alpha_1,y)$ has the same form as of $f(x,y)$.
Moreover, if $f(x,y) = H_1(x,y) H_2(x,y)$, then $f (x+\alpha_1,y) = H_1(x+\alpha_1,y) H_2(x+\alpha_1,y)$, so that $f(x+\alpha_1,y)$ is reducible.

First, consider the case {\bf (A)}.
We take $\alpha_1 = 0$, so that any of other $\alpha_i$'s and $\beta_j$'s is not zero.
Comparing the $y^2$-coefficient and the constant term of $f$ with those of $H_1H_2$, we have $a_3 + a_7 = c_{02} = - 2 (\sigma_4 + \tau_4) = -2 \tau_4$ and $a_3 a_7 = c_{00} = (\sigma_4 - \tau_4)^2 = \tau_4^2$, where we use $\sigma_4 = 0$ by $\alpha_1 = 0$.
Therefore, the elements $a_3$ and $a_7$ are the roots of 
\[
X^2 -c_{02} X +  c_{00} = (X + \tau_4)^2 = (X-c_{02}/2)^2,
\]
so that $a_3 = a_7 = -\tau_4 = c_{02}/2$.
Also from the coefficients of $x y^2$ in $f$ and $H_1H_2$, we have $a_2 + a_6 = c_{12}$, and thus
\[
    \begin{aligned}
        a_2 a_7 + a_3 a_6 = - \tau_4 (a_2+a_6) =  -2\tau_4 (\sigma_3 + \tau_3) 
    \end{aligned}
\]
by \eqref{eq:c12}.
On the other hand, it follows from $c_{10}=a_2 a_7 + a_3 a_6$ and \eqref{eq:c10} together with $\tau_4 \neq 0$ that $\sigma_3 + \tau_3 = - (\sigma_3 - \tau_3)$, whence $\sigma_3 = \alpha_2 \alpha_3 \alpha_4= 0$, a contradiction.

Next, we consider the case {\bf (B)}, and assume $\alpha_1 = 0$.
By \eqref{eq:c60} and \eqref{eq:c00}, we can determine $a_3$ and $a_6$ from the $x^6$-coefficients and the constant terms in $f$ and $H_1H_2$.
Specifically, there are $4$ cases:    
\[
\begin{aligned}
    {\bf (B1)} \
    \begin{cases}
        a_{3} = \sigma_1  - \tau_1,\\
        a_6 = \sigma_4 - \tau_4.
    \end{cases}\quad \ \
    &
    {\bf (B2)} \
    \begin{cases}
        a_{3} = -(\sigma_1 - \tau_1),\\
        a_6 = \sigma_4 - \tau_4.
    \end{cases}\\
     {\bf (B3)} \
    \begin{cases}
        a_{3} = \sigma_1 - \tau_1,\\
        a_6 = -(\sigma_4 - \tau_4).
    \end{cases}
    &
    {\bf (B4)} \
    \begin{cases}
        a_{3} = -(\sigma_1 - \tau_1),\\
        a_6 = -(\sigma_4 - \tau_4).
    \end{cases}
\end{aligned}
\]
Note that $a_6 \neq 0$ by $\alpha_1 = 0$ and $0 = \sigma_4 \neq \tau_4$.
Once the values of $a_3$ and $a_6$ are given, the other $a_i$'s can be also computed by comparing coefficients of $f$ with ones of $H_1H_2$ as follows:
 \[
    \left\{
    \begin{array}{ll}
    a_1 = -(c_{32} - 2a_3)/4 & \mbox{from the coefficients of $x^3 y^2$},\\
    a_2 =  (2 a_4 - a_1^2-c_{22})/4 & \mbox{from the coefficients of $x^2 y^2$},\\
    a_5 = c_{10}/(2a_6) & \mbox{from the coefficients of $x$}.
    \end{array}
    \right.
\]
Here, $a_4$ is determined by $a_4 = c_{50}/(2a_3)$ from the coefficients of $x^5$ if $a_3 \neq 0$, and by $a_4 = \pm \sqrt{c_{40}}$ from $q_1$ below if $a_3 = 0$.
Comparing the other coefficients in $f$ and $H_1 H_2$, we derive the following system of equations:
\[
    \left\{
    \begin{array}{llll}
    q_1 :=& 2 a_3 a_5 + a_4^2 - c_{40} & = 0 & \mbox{from the coefficient of $x^4$},\\
    q_2 := & 2 a_3 a_6 + 2 a_4 a_5 - c_{30} & = 0 & \mbox{from the coefficient of $x^3$},\\
    q_3 := & -2 a_1 a_2 + 2 a_5 - c_{12} & = 0 & \mbox{from the coefficient of $x y^2$},\\
    q_4 := & 2 a_4 a_6 + a_5^2 - c_{20} & = 0 & \mbox{from the coefficient of $x^2$},\\
    q_5 := & -a_2^2 + 2 a_6 - c_{02} & = 0 &  \mbox{from the coefficient of $y^2$}.
    \end{array}
    \right.
\]

We assume $a_3 \neq 0$, and derive a contradiction in each case of {\bf (B1)} -- {\bf (B4)}.
\begin{itemize}
    \item In the case {\bf (B1)}, \textcolor{black}{a straightforward computation with a computer shows}
\[
\left\{
\begin{aligned}
    q_1 =& -4 (\sigma_1 - \tau_1)(\sigma_3 - \tau_3),\\
    q_2 = & 4((\sigma_2 - \tau_2)(\sigma_3 - \tau_3) + (\sigma_1 - \tau_1)(\sigma_4 - \tau_4)),\\
    q_3 =& -\frac{1}{2} (8 \sigma_3 + \tau_1 (\tau_1^2 - 4 \tau_2)),\\
    q_4 = & -4 (\sigma_4 - \tau_4) (\sigma_2 - \tau_2),\\
    q_5 =& -\frac{1}{16}( - 64 \sigma_4 + (\tau_1^2 - 4 \tau_2)^2 ).
\end{aligned}
\right.
\]
From $q_3 = q_5 = 0$, we can represent $\sigma_3$ and $\sigma_4$ as polynomials in $\tau_1$ and $\tau_2$.
Also by $q_1 = 0$, the cases are divided into $3$ cases:
$\sigma_1 = \tau_1$ but $\sigma_3 \neq \tau_3$, $\sigma_1 \neq \tau_1$ but $\sigma_3 = \tau_3$, or $\sigma_1 = \tau_1$ and $\sigma_3 = \tau_3$.
For the first case, it follows from $q_2=0$ that $\sigma_2 = \tau_2$, whence $\phi_1(x)$ is equal to 
\[
    \begin{aligned}
        & x^4 - \tau_1 x^3 +\tau_2 x^2 + \frac{\tau_1(\tau_1^2 - 4 \tau_2)}{8} x + \frac{(\tau_1^2 -4 \tau_2)^2}{64} = {\left(x^2 -\frac{\tau_1}{2} x -\frac{(\tau_1^2-4\tau_2)}{8}\right)^2 }.
    \end{aligned}
\]
This implies $\alpha_i = \alpha_j$ for some $i$ and $j$ with $i \neq j$, a contradiction.
{As for the other two cases, we can derive a contradiction similarly.}

\item In the case {\bf (B4)}, we obtain the same formulae of $q_1$, $q_2$, and $q_4$ as in the case {\bf (B1)}, {and} 
\[
q_5 =  - \frac{1}{16}( (\sigma_1^2 - 4 \sigma_2)^2 - 64 \tau_4 ),\qquad q_3 = - \frac{1}{2} (8 \tau_3 + \sigma_1 (\sigma_1^2 - 4 \sigma_2) )
\]
by a \textcolor{black}{computer} calculation.
Here, $q_1$, $q_2$, $q_3$, $q_4$, and $q_5$ in this case coincide with those in the case {\bf (B1)} if $\sigma_i$ and $\tau_i$ are exchanged for each $i$.
Thus, we can derive a contradiction similarly to the case {\bf (B1)}.
\item For the case {\bf (B2)}, we can factor $q_3$ {as}
\[
q_3 = -\frac{1}{2}(\alpha_1 - \alpha_2 - \alpha_3 + \alpha_4)(\alpha_1 - \alpha_2 + \alpha_3 - \alpha_4) (\alpha_1 + \alpha_2 - \alpha_3 - \alpha_4),
\]
whence $\alpha_1 + \alpha_2 = \alpha_3 + \alpha_4$, $\alpha_1 + \alpha_3 = \alpha_2 + \alpha_4$, or $\alpha_1 + \alpha_4 = \alpha_2 + \alpha_3$.
A tedious computation with a \textcolor{black}{computer} shows that
\[
q_5 = 
\left\{
\begin{array}{ll}
    - (\alpha_2 - \alpha_3)^2 (\alpha_2 - \alpha_4)^2 & \mbox{if $\alpha_1 + \alpha_2 = \alpha_3 + \alpha_4$},\\
     -(\alpha_2 - \alpha_3)^2 (\alpha_3 - \alpha_4)^2 & \mbox{if $\alpha_1 + \alpha_3 = \alpha_2 + \alpha_4$},\\
      -(\alpha_2 - \alpha_4)^2 (\alpha_3 - \alpha_4)^2 & \mbox{if $\alpha_1 + \alpha_4 = \alpha_2 + \alpha_3$},
\end{array}
\right.
\]
each of which should not be zero by our assumption $\alpha_i \neq \alpha_j$ for any $i \neq j$.
\item Also for the case {\bf (B3)}, we can {obtain} factorization formulae similar to ones in the case {\bf (B2)}, which derive a contradiction to our assumption $\beta_i \neq \beta_j$ for any $i \neq j$.
\end{itemize}

Finally, we assume $a_3 = 0$ (and $\alpha_1 = 0$ continuously), namely $\sigma_1 = \tau_1$.
In this case, it follows from \eqref{eq:c40} that $c_{40} = (\sigma_2 - \tau_2)^2$, whence we can take $a_4 = \pm (\sigma_2 - \tau_2)$.
From this together with $a_6 = \pm (\sigma_4 - \tau_4)$, the case is divided into four cases.
In each of the four cases, we can derive a contradiction by a tedious computation, similarly to the cases {\bf (B1)} -- {\bf (B4)} with $a_3 \neq 0$.

\end{proof}

\begin{remark}
    In the proof of Proposition \ref{prop:irr}, we used Magma~\cite{bosma1997magma} for some symbolic computations, which can be of course conducted by hand or by other computer algebra systems (we did not use any function specific to Magma).
\end{remark}

\section{Singularity analysis and concrete examples}

We use the same notation as in the previous section; for example, let $k$ be an algebraically closed field of characteristic $p$ with $p =0$ or $p \geq 5$.
We also denote by $\tilde{C}$ the projective closure of $C$ in $\mathbb{P}^2$.
In this section, the possible number of singular points on $\tilde{C}$ is determined.
Some concrete examples are also provided.

\subsection{Singularity analysis}\label{subsec:sing}

Let $F$ be the homogenization of $f$ by an extra variable $z$, say 
\[
\begin{aligned}
F=& (c_{60} x^6 + c_{42} x^4 y^2) + (c_{50} x^5 + c_{32} x^3 y^2) z + (c_{40} x^4 + c_{22}x^2 y^2 + c_{04}y^4) z^2 \\
& + (c_{30} x^3 + c_{12}x y^2) z^3 + (c_{20} x^2 + c_{02} y^2)z^4 + c_{10} xz^5 + c_{00}z^6.
\end{aligned}
\]
Then, $\tilde{C}$ is the locus $F=0$ in $\mathbb{P}^2 = \mathrm{Proj}(k[x,y,z])$.
By degree-genus formula, the arithmetic genus of $\tilde{C}$ is $g_a(\tilde{C})=10$, and thus $\tilde{C}$ has a singular point.
Let $\mathrm{Sing}(\tilde{C})$ denote the set of singular points on $\tilde{C}$ in $\mathbb{P}^2$, namely
\[
\mathrm{Sing}(\tilde{C}) = \{ P \in \mathbb{P}^2 : F(P)= F_x(P) = F_y(P) = F_z(P) = 0 \},
\]
where $F_x := \frac{\partial F}{\partial x}$, $ F_y:= \frac{\partial F}{\partial y}$, and $F_z:=\frac{\partial F}{\partial z}$.
Note that it suffices to consider the locus $F_x = F_y = F_z = 0$ by Euler's relation $\mathrm{deg}(F) F = x F_x + y F_y + z F_z$ with $\mathrm{deg}(F) = 6$ and $p \neq 2,3$.
Singular points $(x:y:z)$ with $z=1$ are called affine singular points, while those with $z=0$ are singular points at infinity.
In the following, we shall determine the singular locus $\mathrm{Sing}(\tilde{C})$ explicitly, by an arithmetic method.

First, denoting by $m_{P}$ the multiplicity of $\tilde{C}$ at a point $P$, we have 
\begin{equation}\label{inequality:genus-multiplicity}
g(H) \le g_a(\tilde{C}) - \sum_{P\in \mathrm{Sing}(\tilde{C})} \frac{m_{P}(m_{P}-1)}{2},
\end{equation}
so that $\tilde{C}$ has at most $5$ singular points.

We start with determining singular points at infinity and their multiplicities.

\begin{lemma}\label{lem:sing_inf}
With notation as above, the singular points on $\tilde{C}$ at infinity are $(1:0:0)$ or $(0:1:0)$ only.
Among the two points, $(0:1:0)$ is always a singular point on $\tilde{C}$, whereas $(1 : 0 : 0) \in \mathrm{Sing}(\tilde{C})$ if and only if $c_{60}=c_{50} = 0$, equivalently $\sigma_1 = \tau_1$ $($i.e., $\alpha_1 + \alpha_2 + \alpha_3 + \alpha_4 =\beta_1 + \beta_2 + \beta_3 + \beta_4)$.

Moreover, each of the singularities $(1:0:0)$ and $(0:1:0)$ on $\tilde{C}$ has multiplicity $2$.
\end{lemma}

\begin{proof}
It is straightforward that $F_x(x,y,0) = 6 c_{60} x^5 + 4 c_{42} x^3 y^2$, $F_y(x,y,0)= 2 c_{42} x^4 y$, and $F_z(x,y,0) = c_{50} x^5 + c_{32} x^3 y^2$
with $c_{42}=-4\neq 0$, from which the assertions hold clearly.
The assertion on the multiplicity follows from $F_{yy}(x,y,0) =2c_{42} x^4$ and $F_{zz}(0,y,0) = 2c_{04} y^4$ together with $c_{42} = -4 \neq 0$ and $c_{04} = 1 \neq 0$.
\end{proof}

Next, consider affine singular points.
Recall from \eqref{eq:sextic} that
\begin{align}
    f_x =& -2(\phi_1'(x) + \phi_2'(x)) y^2 + 2 (\phi_1(x) - \phi_2(x)) (\phi_1'(x) - \phi_2'(x)), \label{eq:fx}\\
    f_y =& 4 y^3 - 4 (\phi_1(x) + \phi_2(x))y. \label{eq:fy}
\end{align}

\begin{lemma}\label{lem:sing_y_neq_0}
    With notation as above, there is no singular point on $\tilde{C}$ of the form $(x: y :1)$ with $y \neq 0$.
    Moreover, if $(x:0:1)$ is a singular point on $\tilde{C}$, then it has multiplicity $2$.
\end{lemma}

\begin{proof}
Assume for a contradiction that $\tilde{C}$ has a singular point $(x:y:1)$ with $y \neq 0$.
By \eqref{eq:fy}, we have $y^2 = \phi_1(x) + \phi_2(x)$.
Substituting this into $y^2$ in \eqref{eq:sextic} and \eqref{eq:fx}, we obtain
$f = - 4\phi_1(x) \phi_2(x) = 0$ and
\[
f_x = -4 (\phi_1'(x) \phi_2(x) + \phi_1(x) \phi_2'(x)) = -4 (\phi_1(x)\phi_2(x))'=0.
\]
This contradicts that $\phi_1 \phi_2$ has no double root.

As for the assertion on the multiplicity, assume for a contradiction that $(x:0:1)$ is a singular point on $\tilde{C}$ with multiplicity $\geq 3$.
Then it follows from $f_{yy} = 12 y^2 - 4(\phi_1 (x) + \phi_2 (x)) = 0$ together with $f(x,0) = (\phi_1(x)  -\phi_2(x))^2$ that $\phi_1$ and $\phi_2$ have a common root, which contradicts our assumption that $\alpha_i$'s and $\beta_j$'s are mutually distinct.

\end{proof}



By Lemma \ref{lem:sing_y_neq_0}, it suffices to consider singular points of the form $(x:0:1)$.
For $y=0$, it follows from $\mathrm{deg}_y F \geq 2$ that $F_y (x,0,1)=0$, and we have
\[
\begin{aligned}
    F (x,0,z) =& z^6 (\phi_1 (x/z) - \phi_2(x/z))^2\\
    =& ( (\sigma_1 - \tau_1) x^3 - (\sigma_2 - \tau_2) x^2z + (\sigma_3 - \tau_3) x z^2- (\sigma_4 - \tau_4)z^3)^2 
\end{aligned}
\]
from \eqref{eq:sextic}.
Since $F(x,0,1) =f(x,0)$ and $\frac{d}{dx}(F(x,0,1)) = F_x(x,0,1)$, the number of affine singular points is equal to that of common roots of $h_1^2$ and $(h_1^2)' = 2 h_1 h_1'$, where
\[
\begin{aligned}
    h_1 :=& -(\phi_1 - \phi_2) = (\sigma_1 - \tau_1)x^3 - (\sigma_2 - \tau_2) x^2 + (\sigma_3 - \tau_3) x - (\sigma_4-\tau_4),\\
    h_1' :=& \frac{d}{dx}h_1 = -(\phi_1' - \phi_2') = 3(\sigma_1 - \tau_1)x^2 - 2(\sigma_2 - \tau_2) x + (\sigma_3 - \tau_3).
\end{aligned}
\]
From this, we obtain the following proposition:
\begin{proposition}\label{prop:sing1}
    With notation as above, assume that $\sigma_1 - \tau_1 \neq 0$.
    \begin{enumerate}
        \item[{\rm (I-1)}] If $\mathrm{Res}_x \left( h_1, \frac{d}{dx} h_1 \right)  \neq 0$, then $\tilde{C}$ has exactly $3$ affine singular points.
        In this case, we conclude $\# \mathrm{Sing}(\tilde{C})=4$ and
        \[
        \mathrm{Sing}(\tilde{C})=\{(0:1:0), (\xi_1:0:1), (\xi_2:0:1), (\xi_3:0:1)\}
        \]
        for the $3$ simple roots $\xi_1$, $\xi_2$, and $\xi_3$ of $h_1$.
        \item[{\rm (I-2)}] If $\mathrm{Res}_x \left( h_1, \frac{d}{dx} h_1 \right)  = 0$ and $\mathrm{Res}_x \left( \frac{d}{dx} h_1, \frac{d^2}{dx^2} h_1  \right)  \neq 0$, then $\tilde{C}$ has exactly $2$ affine singular points.
        In this case, we conclude $\# \mathrm{Sing}(\tilde{C})=3$ and
        \[
        \mathrm{Sing}(\tilde{C})=\{(0:1:0), (\xi_1:0:1), (\xi_2:0:1) \}
        \]
        for the double root $\xi_1$ and the simple root $\xi_2$ of $h_1$.
        \item[{\rm (I-3)}] $\tilde{C}$ has a unique singular point of the form $(x:0:1)$ if and only if $\mathrm{Res}_x ( h_1, \frac{d}{dx} h_1 ) = \mathrm{Res}_x ( \frac{d}{dx} h_1, \frac{d^2}{dx^2} h_1 ) = 0$.
        In this case, such the unique singularity is $( \xi :0:1)$ with $\xi := (\sigma_2-\tau_2)/3(\sigma_1-\tau_1)$, and hence we conclude
        \[
        \mathrm{Sing}(\tilde{C})= \{ (0:1:0), ( \xi :0:1) \}.
        \]
    \end{enumerate}     
    In each of the three cases, each singularity has multiplicity $2$ by Lemma \ref{lem:sing_y_neq_0}.
\end{proposition}

Next, we consider the case where $\sigma_1 - \tau_1 = 0$; in this case, it follows from Lemma \ref{lem:sing_inf} that $\mathrm{Sing}(\tilde{C}) \supset \{ (0:1:0), (0:0:1)\}$.
Dividing the cases into $\sigma_2-\tau_2 \neq 0$ or $\sigma_2-\tau_2 = 0$, we obtain the following proposition:

\begin{proposition}\label{prop:sing2}
With notation as above, assume $\sigma_1 = \tau_1$ and $\sigma_2-\tau_2 \neq 0$.
Then we have the following:
\begin{enumerate}
    \item[{\rm (II-1)}] If $\mathrm{Res}_x(h_1, \frac{d}{dx}h_1) \neq 0$, then $\tilde{C}$ has exactly $2$ singularities of the form $(\xi:0:1)$, which are given by the distinct $2$ roots $\xi$ of $h_1$.
    In this case, we conclude $\# \mathrm{Sing}(\tilde{C})=4$.
    \item[{\rm (II-2)}] If $\mathrm{Res}_x(h_1, \frac{d}{dx}h_1) = 0$, then $\tilde{C}$ has a unique singular point of the form $(\xi:0:1)$, which is given by $\xi = (\sigma_3-\tau_3)/2(\sigma_2-\tau_2)$.
    In this case, we conclude $\# \mathrm{Sing}(\tilde{C})=3$.
\end{enumerate}
Moreover, supposing $\sigma_1 = \tau_1$ and $\sigma_2=\tau_2$, we have the following:
\begin{enumerate}
    \item[{\rm (II-3)}] If $\sigma_3 - \tau_3 \neq 0$, then $\tilde{C}$ has a unique singular point of the form $(\xi:0:1)$, which is given by $\xi = (\sigma_4-\tau_4)/(\sigma_3-\tau_3)$.
    In this case, we conclude $\# \mathrm{Sing}(\tilde{C})=3$ and $\mathrm{Sing}(\tilde{C}) = \{ (\xi:0:1), (0:1:0), (1:0:0) \}$.
    \item[{\rm (II-4)}] If $\sigma_3= \tau_3$, then $\tilde{C}$ has no singular point of the form $(x:0:1)$, whence we conclude $\# \mathrm{Sing}(\tilde{C})=2$ and $\mathrm{Sing}(\tilde{C}) = \{ (0:1:0), (1:0:0) \}$.
\end{enumerate}
In each of the above three cases, each singularity has multiplicity $2$.
Note also that, in the case (II-4), we used $\sigma_4 \neq \tau_4$ which comes from that $\alpha_i$'s and $\beta_j$'s are mutually distinct.
\end{proposition}

Considering results obtained above together, we obtain Theorem \ref{thm:main2}, see Table \ref{table:sing} below for a summary.
With the exception of the loci $\sigma_1 -\tau_1 = 0$ and $\mathrm{Res}_x (h_1, \frac{d}{dx} h_1) = 0$, the point $(\alpha_1, \alpha_2,\alpha_3,\alpha_4, \beta_1,\beta_2,\beta_3,\beta_4)$ produces a plane sextic model $\tilde{C}$ with $4$ double points, $3$ of which are affine, while the other one is at infinity.
Namely, $\tilde{C}$ has such $4$ singularities generically.

\renewcommand{\arraystretch}{1.5}
\begin{center}
    \begin{table}[h]
    \centering
    \caption{The number of singular points on the projective closure $\tilde{C}$ of a plane sextic curve $C:f(x,y)=0$ associated with $(\alpha_1,\alpha_2,\alpha_3,\alpha_4,\beta_1,\beta_2,\beta_3,\beta_4)\in k^8$, where $\alpha_i$'s and $\beta_j$'s are mutually distinct elements.
    For each integer $i$ with $1\leq i \leq 4$, the degree-$i$ elementary symmetric polynomial on $\alpha_1,\alpha_2,\alpha_3,\alpha_4$ and that on $\beta_1,\beta_2,\beta_3,\beta_4$ are denoted by $\sigma_i$ and $\tau_i$ respectively.
    We also set $h_1 := (\sigma_1-\tau_1)x^3 - (\sigma_2 - \tau_2) x^2 + (\sigma_3 - \tau_3) x - (\sigma_4 - \tau_4)$.
    The notation ``$(m,n)$'' means $\tilde{C}$ has $m$ affine singularities $(x:y:1)$ and $n$ singularities $(x:y:0)$ at infinity.
    Each ``Type'' corresponds to one given in Propositions \ref{prop:sing1} and \ref{prop:sing2}.}\label{table:sing}
    \vspace{2mm}
    \begin{tabular}{|c|c|c|c|c|c|}
    \hline
         \multicolumn{3}{|c|}{Equivalent conditions} & {$(m,n)$} & $\# \mathrm{Sing}(\tilde{C})$ & {Type} \\
         \hline
          \multirow{3}{*}{$\sigma_1 \neq \tau_1$} & \multicolumn{2}{c|}{$\mathrm{Res}_x(h_1,h_1') \neq 0$} & $(3,1)$ & $4$ & I-1  \\ \cline{2-6}
          &  \multirow{2}{*}{$\mathrm{Res}_x(h_1,h_1') = 0$} & $\mathrm{Res}_x(h_1',h_1'') \neq 0$ & $(2,1)$ & $3$ & I-2 \\ \cline{3-6}
          &   & $\mathrm{Res}_x(h_1',h_1'') = 0$ & $(1,1)$ & $2$ & I-3 \\ \hline
         \multirow{4}{*}{$\sigma_1=\tau_1$} &  \multirow{2}{*}{$\sigma_2\neq \tau_2$} & $\mathrm{Res}_x(h_1,h_1') \neq 0$ & $(2,2)$ & $4$ & II-1 \\ \cline{3-6}
          &   & $\mathrm{Res}_x(h_1,h_1') = 0$ & $(1,2)$ & $3$ & II-2 \\ \cline{2-6}
          & \multirow{2}{*}{$\sigma_2 = \tau_2$} & $\sigma_3 \neq \tau_3$ & $(1,2)$ & $3$ & II-3  \\ \cline{3-6}
          &  & $\sigma_3 = \tau_3$ & $(0,2)$ & $2$ & II-4 \\ \hline
    \end{tabular}
\end{table}
\end{center}

\begin{remark}
If $\sigma_1 - \tau_1= \sigma_2 - \tau_2 = 0$, then $f$ is simply written as
\begin{equation}\label{eq:simplified1}
\begin{aligned}
    f =& -4 x^4 y^2 + 4\sigma_1 x^3 y^2 +(-4 \sigma_2 x^2y^2 +y^4)  \\
    & +2({\sigma_3} +\tau_3)xy^2 +((\sigma_3-\tau_3)^2x^2 -2 (\sigma_4 + \tau_4) y^2)\\
    & -2 (\sigma_3-\tau_3)(\sigma_4-\tau_4) x + (\sigma_4-\tau_4)^2.\\
\end{aligned}
\end{equation}
Moreover, if $\sigma_1 - \tau_1= \sigma_2 - \tau_2 = \sigma_3 - \tau_3 = 0$, we obtain a more simplified form
\begin{equation}\label{eq:simplified2}
\begin{aligned}
        f =& -4 x^4 y^2 + 4 \sigma_1 x^3y^2 + (-4 \sigma_2 x^2 y^2 + y^4)\\
    &+4\sigma_3 x y^2 -2 (\sigma_4 + \tau_4) y^2 + (\sigma_4-\tau_4)^2.
\end{aligned}
\end{equation}
\end{remark}

\subsection{Concrete examples}

In this subsection, we show a concrete example for each type of singularities.
As noted in Remark \ref{rem:3points}, we may fix $3$ among $\alpha_i$'s and $\beta_j$'s.
For simplicity, we take $(\alpha_1,\alpha_2,\alpha_3) = (0,1,-1)$ and put $\alpha = \alpha_4$ in examples below.

\begin{example}[Type I: $\sigma_1-\tau_1\neq 0$]
    Let $k=\overline{\mathbb{F}}_{31}$, and let $C:f(x,y)=0$ be our sextic curve associated with a point $(\alpha,\beta_1, \beta_2, \beta_3, \beta_4) \in k^5$.
\begin{enumerate}
    \item[(I-1)] For $(\alpha,\beta_1, \beta_2, \beta_3, \beta_4) =(20,28,16,7,27) \in \mathbb{F}_{31}^5$, we have $\sigma_1 - \tau_1  = 4 \neq 0$, and $\mathrm{Res}_x ( h_1, \frac{d}{dx} h_1 ) = 27 \neq 0$.
    In this case, the computed sextic is
    \[
    \begin{aligned}
           f = & (16 x^6 + 27 x^4 y^2) + (22 x^5  + 10 x^3 y^2) +(23 x^4 + 14 x^2 y^2 + y^4)  \\
           & + (13 x^3 + 29xy^2) + (16x^2  + 9y^2) + 10 x + 28,
    \end{aligned}
    \]
    and there are exactly $4$ singular points on the projective closure $\tilde{C}$ of $C$:
    $(24:0:1)$, $(4:0:1)$, $(12:0:1)$, and $(0:1:0)$.
    Each of them is of multiplicity two.
    
    \item[(I-2)] For $(\alpha,\beta_1, \beta_2, \beta_3, \beta_4) =(11,2,13,29,22) \in \mathbb{F}_{31}^5$, one has $\sigma_1-\tau_1  = 7 \neq 0$, and $\mathrm{Res}_x ( h_1, \frac{d}{dx} h_1 ) = 0$, but $\mathrm{Res}_x ( \frac{d}{dx} h_1, \frac{d^2}{dx^2} h_1 ) = 5 \neq 0$.
    In this case, the computed sextic is
    \[
    \begin{aligned}
           f = & (18 x^6 + 27 x^4 y^2 ) + (25 x^5 + 30 x^3 y^2 )  + (24 x^4 + 27 x^2 y^2 + y^4)\\
           & + (20 x^3+ 8 xy^2) + (18x^2 + 25y^2) +30x + 9,
    \end{aligned}
    \]
    and there are exactly $3$ singular points on the projective closure $\tilde{C}$ of $C$:
    $(25:0:1)$, $(7:0:1)$, and $(0:1:0)$.
    Each of them is of multiplicity two.
    
    \item[(I-3)] For $(\alpha,\beta_1, \beta_2, \beta_3, \beta_4) =(7,2,5,8,19) \in \mathbb{F}_{31}^5$, we examine $\sigma_1 - \tau_1  = 4 \neq 0$, and moreover $\mathrm{Res}_x ( h_1, \frac{d}{dx} h_1 ) = \mathrm{Res}_x ( \frac{d}{dx} h_1, \frac{d^2}{dx^2} h_1 ) = 0$.
    Therefore, the computed sextic is
    \[
    \begin{aligned}
            f = & (16 x^6 + 27 x^4 y^2) + (26 x^5  + 20 x^3 y^2) + (26 x^4 + 13 x^2 y^2 + y^4)\\
            & + (18 x^3 + 19 xy^2)+ (24 x^2 + 29 y^2) + 15 x + 1,
    \end{aligned}
    \]
    and there are exactly $2$ singular points on the projective closure $\tilde{C}$ of $C$:
    $(12:0:1)$ and $(0:1:0)$.
    Each of them is of multiplicity two.
    
     
\end{enumerate}
\end{example}

\begin{example}[Type II: $\sigma_1-\tau_1= 0$]
    Let $k=\overline{\mathbb{F}}_{31}$, and let $C:f(x,y)=0$ be our sextic curve associated with a point $(\alpha,\beta_1, \beta_2, \beta_3, \beta_4) \in k^5$.
\begin{enumerate}
    \item[(II-1)] For $(\alpha,\beta_1, \beta_2, \beta_3, \beta_4) =(8,12,26,28,4) \in \mathbb{F}_{31}^5$, one can examine $\sigma_1 = \tau_1$, and moreover $\sigma_2-\tau_2=2 \neq 0$ and $(\sigma_3-\tau_3)^2 -4 (\sigma_2-\tau_2) (\sigma_4-\tau_4) =14 \neq 0$.
    Therefore, the computed sextic is
    \[
    f = 27 x^4 y^2 + x^3 y^2 + (4 x^4 + 8 x^2 y^2+y^4) + (14 x^3  + 6xy^2)  + (23x^2 + 17y^2)  + 13x + 18,
    \]
    and there are exactly $4$ singular points on the projective closure $\tilde{C}$ of $C$:
    $(14:0:1)$, $(23:0:1)$, $(0:1:0)$, and $(1:0:0)$.
    Each of them is of multiplicity two.
    
    \item[(II-2)] For $(\alpha,\beta_1, \beta_2, \beta_3, \beta_4) =(5,2,10,26,29) \in \mathbb{F}_{31}^5$, one can examine $\sigma_1 = \tau_1$, and moreover $\sigma_2-\tau_2=22 \neq 0$ and $(\sigma_3-\tau_3)^2 -4 (\sigma_2-\tau_2) (\sigma_4-\tau_4) =0$.
    Therefore, the computed sextic is
    \[
    f = 27x^4y^2 + 20x^3y^2 + (19x^4 + 17x^2y^2  + y^4)+ (22x^3 + 12xy^2) + (12x^2 + 3y^2)  + 17x + 10,
    \]
    and there are exactly $3$ double points on the projective closure of $\tilde{C}$:
    One is an affine singular point $(25:0:1)$, and the others are $(0:1:0)$ and $(1:0:0)$, which are points at infinity.
     
    \item[(II-3)] Put $(\alpha,\beta_1, \beta_2, \beta_3, \beta_4) =(29,2,7,14,6) \in \mathbb{F}_{31}^5$.
    It follows that $\sigma_1 = \tau_1$, and moreover $\sigma_2=\tau_2$ but $\sigma_3-\tau_3 = 20 \neq 0$.
    The computed sextic is
    \[
    f = 27x^4y^2 + 23x^3y^2 + (4x^2y^2+ y^4 ) + 30xy^2  + (28x^2 + 4y^2) + 13x+ 4
    \]
    as in \eqref{eq:simplified1}, and the projective closure of $\tilde{C}$ has exactly $3$ double points.
    One is an affine singular point $(28:0:1)$, and the others are $(0:1:0)$ and $(1:0:0)$, which are points at infinity.

    \item[(II-4)] For $(\alpha,\beta_1, \beta_2, \beta_3, \beta_4) =(2,8,20,24,12) \in \mathbb{F}_{31}^5$, one can examine $\sigma_1 = \tau_1$, and moreover $\sigma_2-\tau_2=\sigma_3-\tau_3=0$ and $\sigma_4-\tau_4 =17 $.
    Therefore, the computed sextic is
    \[
    f = 27 x^4 y^2 + 8 x^3 y^2 + (4 x^2 y^2 + y^4) + 23 x y^2  + 3 y^2 + 10
    \]
    as in \eqref{eq:simplified2}, and there are exactly $2$ singularities on the projective closure of $\tilde{C}$:
    One is $(0:1:0)$, and the others are $(1: 0 : 0)$.
    Each of them is of multiplicity two.
\end{enumerate}
\end{example}

\section{Concluding remark}\label{sec:conc}

In this paper, we focused on genus-$5$ non-hyperelliptic Howe curves, which are constructed as non-singular curves birational to fiber products of two hyperelliptic curves $C_1$ and $C_2$ of genera $g_1$ and $g_2$ sharing precisely $r$ ramification points in $\mathbb{P}^1$, for $(g_1,g_2,r) = (2,2,4)$ or $(1,1,0)$.
While the former case was treated in \cite{MK23}, we studied the latter case in this paper.
Specifically, we presented an explicit plane sextic model for Howe curves in the case;
we proved that the associated sextic polynomial is absolutely irreducible.
This sextic can be computed easily (in fact, in constant time) once the ramification points of $C_1$ and $C_2$ are specified.
We also determined the possible number of singularities on the sextic together with concrete forms of the singularities, and we found that there are $4$ double points generically.
These together with results in \cite{MK23} imply the existence of a genus-$5$ non-hyperelliptic curve $H$ with $\mathrm{Aut}(H) \supset \mathbf{V}_4$ such that its associated plane sextic has exactly $s$ double points, for any $s \in \{ 2,3,4,5 \}$.

Our sextic model constructed in this paper would be feasible to analyze non-hyperelliptic Howe curves of genus five as plane singular curves, by constructing their function fields.
For example, one can determine whether two such curves are isomorphic to each other or not, which can be applied to enumerating the isomorphism classes of curves (defined over finite fields) such as superspecial ones.
We leave this kind of applications our future work.
Another interesting open problem is to provide representable families for non-hyperelliptic Howe curves of genus $5$ (or more generally non-hyperelliptic curves $H$ of genus $5$ with $\mathrm{Aut}(H) \supset \mathbf{V}_4$), as in \cite{lercier2014parametrizing} (resp.\ \cite{BG}) for genus-$3$ (resp.\ genus-$5$) cases.
As we also noted in Remark \ref{rem:gen}, our method to construct a sextic model in this paper could be extended to the case of Howe curves of other genus associated with two hyperelliptic curves sharing no ramification points, and its formulation is also an interesting problem.

\subsection*{Acknowledgements}
The author thanks Kazuhiro Yokoyama for helpful comments.
This work was supported by JSPS Grant-in-Aid for Young Scientists 20K14301 and 23K12949.

\bibliography{ref}
\bibliographystyle{plain}

\end{document}